\documentclass[10pt]{article}

\usepackage{amsmath,amscd,amsfonts,amsthm,amssymb}

\newcommand{\shrinkmargins}[1]{
  \addtolength{\textheight}{#1\topmargin}
  \addtolength{\textheight}{#1\topmargin}
  \addtolength{\textwidth}{#1\oddsidemargin}
  \addtolength{\textwidth}{#1\evensidemargin}
  \addtolength{\topmargin}{-#1\topmargin}
  \addtolength{\oddsidemargin}{-#1\oddsidemargin}
  \addtolength{\evensidemargin}{-#1\evensidemargin}
  }

\shrinkmargins{.7}

\DeclareMathOperator{\SL}{SL}

\DeclareMathOperator{\Spec}{Spec}

\DeclareMathOperator{\Gal}{Gal}

\DeclareMathOperator{\Ind}{Ind}
\DeclareMathOperator{\Res}{Res}

\DeclareMathOperator{\rank}{rank}
\DeclareMathOperator{\Aut}{Aut}

\newcommand{\field}[1]{\mathbb{#1}}

\newcommand{\Qp}{\field{Q}_p}
\newcommand{\Q}{\field{Q}}

\newcommand{\Zp}{\field{Z}_p}
\newcommand{\Z}{\field{Z}}
\newcommand{\F}{\field{F}}
\newcommand{\Fp}{\field{F}_p}

\newcommand{\C}{\field{C}}
\renewcommand{\P}{\field{P}}

\newcommand{\XX} {\mathcal{X}}
\newcommand{\ra}{\rightarrow}

\newcommand{\OO}{\mathcal{O}}

\newcommand{\tensor} {\otimes}

\newcommand{\set}[1]{\{#1\}}

\newcommand{\beq}{\begin{displaymath}}
\newcommand{\eeq}{\end{displaymath}}
\newcommand{\beqn}{\begin{equation}}
\newcommand{\eeqn}{\end{equation}}

\theoremstyle{plain}
\newtheorem{thm}{Theorem}[section]
\newtheorem{prop}[thm]{Proposition}

\newtheorem{lem}[thm]{Lemma}
\newtheorem*{intro}{Theorem}

\theoremstyle{definition}

\newtheorem{exmp}[thm]{Example}

\theoremstyle{remark}

\title{Upper bounds for the growth of Mordell-Weil ranks in pro-$p$ towers of Jacobians}

\author{Jordan S. Ellenberg}

\begin{document}

\maketitle

\begin{abstract}
We study the variation of Mordell-Weil ranks in the Jacobians of curves in a pro-$p$ tower over a fixed number field.  In particular, we show that under mild conditions the Mordell-Weil rank of a Jacobian in the tower is bounded above by a constant multiple of its dimension.  In the case of the tower of Fermat curves, we show that the constant can be taken arbitrarily close to $1$.  The main result is used in the forthcoming paper of Guillermo Mantilla-Soler on the Mordell-Weil rank of the modular Jacobian $J(Np^m)$.
\end{abstract}

\section{Introduction}

\label{s:intro}

Let $F$ be a number field, and let $X/F$ be a smooth algebraic curve, not
necessarily proper. Let $p$ be a prime.

By a {\em pro-$p$ tower} of curves over $X$ we mean a diagram
\beq
\ldots \ra X_3/F \ra X_2/F \ra X_1/F \ra X_0/F = X
\eeq
where the map $X_n \ra X$ is an \'{e}tale cover, geometrically Galois with Galois group a finite $p$-group.  Let $T$ be a finite set of primes of $\OO_F$ including $p$ and $\infty$.  We say the tower {\em has good reduction} away from $T$ if the diagram above extends to
\beq
\ldots \ra \XX_3/R \ra \XX_2/R \ra \XX_1/R \ra \XX_0/R
\eeq
where $R = \OO_F[1/T]$, each $\XX_n$ is smooth over $R$, and the maps are again finite \'{e}tale.\footnote{A better definition might be to say that the tower has good reduction away from $T$ if $X/\Spec \OO_F[1/T]$ is the complement in a proper smooth curve $\tilde{X} / \Spec \OO_F[1/T]$ of a divisor $D$ which is also proper and smooth over $\Spec \OO_F[1/T]$; but the definition used here avoids technical issues orthogonal to the theme of the paper.}  Write $G_F$ for the absolute Galois group of $F$ and $G_T(F)$ for the Galois group of the maximal extension of $F$ unramified outside $T$.

The main theorem of this paper is that the Mordell-Weil ranks of the Jacobians of the curves in a tower grow at most {\em linearly} in the genus.

\begin{intro}[Theorem \ref{th:main}] Write $\tilde{X}_n/F$ for the smooth proper curve over $F$ containing $X_n$ as open subscheme.  Suppose $G_F$ acts trivially on $H_1(X_{\bar{F}},\Z/ p\Z)$. Then  for any $n$ and any subquotient $A$ of the Jacobian of $\tilde{X}_n$ we have
\beq
\rank A(F) \leq 2 \dim_{\F_p} H^1(G_T(F), \F_p) \dim A.
\eeq
\end{intro}

The upper bound for $\rank A(F) / \dim A$ in Theorem~\ref{th:main} depends on $X$ and $p$.  In the special case of Fermat curves over $\Q$, we can do better.

\begin{intro}[Theorem \ref{th:fermatrank}, special case $K=\Q$]  Let $p$ be an odd prime, and let $J_n$ be the Jacobian of the plane curve with equation
\beq
x^{p^n} + y^{p^n} + z^{p^n} = 0
\eeq
Then there is a constant $C'$, depending only on $p$, such that 
\beq
\rank J_n(\Q) \leq \dim J_n + C'p^n.
\eeq
\end{intro}

We note that $\dim J_n$ is on order $p^{2n}$.

\subsection{Motivation and history}

The original motivation for this paper comes from the theorem of Chabauty, and its improvement by Coleman~\cite{cole:chabauty}, from which one obtains an upper bound for the number of points on $X(F)$ under the condition that the (generalized) Jacobian $J$ of $X$ satisfies the bound
\begin{equation}
\rank J(F) < \dim J.
\label{eq:chabauty}
\end{equation}
More recent work of Bruin, Flynn, Wetherell, and others has made it clear that Chabauty's method can be applied even when the inequality \eqref{eq:chabauty} fails for $X$, if it can be shown that is is satisfied for some set of \'{e}tale covers of $X$.  This leads naturally to the question
\begin{quote}
{\bf Question A:} Does every curve $X/F$ admit an \'{e}tale cover $Y/F$ such that $\rank J(Y)(F) < \dim J(Y)$?
\end{quote}

Any version of our main theorem giving a bound
\beq
\rank A(F) \leq C \dim A + o(\dim A)
\eeq
with $C < 1$ would imply an affirmative answer to Question A.  Unfortunately, the constant we achieve is always at least $1$; in the most favorable case, that of Fermat curves, we obtain a constant of $1$, exactly at the Chabauty boundary.

The idea of this paper might be seen as an attempt at a ``non-abelian Chabauty" method.  For simplicity, we explain this in the case where $X = \P^1-\set{0,1,\infty}$ and $\set{X_n}$ is the tower of Fermat curves.  Instead of studing the Selmer groups of the $J_n$ individually, one might try to work ``at the top" by studying a direct limit in $m,n$ of $H^1(G_T(\Q(\zeta_{p^m})), J_n[p^\infty])$.   This limit will be a cofinitely generated module for the non-abelian Iwasawa algebra $\Lambda = \Z_p[[\Gamma]]$, where $\Gamma = H_1(X_{\bar{F}},\Z_p) \rtimes \Gal(\Q(\zeta_{p^\infty})/\Q)$.  One might hope that a ``Selmer module" of this kind could be shown to be small in an appropriate sense, leading to an affirmative answer to Question A.  The difficulty, of course, is that we have not placed any conditions on the restriction of our cohomology classes to decomposition groups at $p$.  It is not clear to us exactly what conditions might be appropriate.

None of the Iwasawa-theoretic machinery is used in the present paper, but the argument should nonetheless be thought of as a finite-level approximation to the approach sketched above.

The main text of the present paper was written in 2002, at which time the author gave several lectures about the material contained here; we apologize for the long delay in making it publicly available.  The intervening period saw the appearance of the beautiful work of Minhyong Kim~\cite{kim:selmer}, which can also be seen as a kind of non-abelian Chabauty -- in his work, the quotient of the geometric etale fundamental group $\pi$ by the $m$th term of its lower central series takes on the role played by $\pi^{ab} = H_1(X_{\bar{F}},\Z_p)$ in the classical method, and by $\pi/[[\pi,\pi],[\pi,\pi]]$ in the above paragraph.  Under widely believed conjectures on Galois representations (and unconditionally in the case $X = \P^1-\set{0,1,\infty}$) Kim can show that the ``unipotent analogue" of \eqref{eq:chabauty} is satisfied, and the finiteness of $X(F)$ follows.

The impetus for releasing the paper now is the recent work of Guillermo Mantilla-Soler~\cite{mant:thesis}, which uses Theorem~\ref{th:main} as an ingredient in an upper bound for the Mordell-Weil rank of the modular Jacobian $J(Np^m)$, as $m$ grows with $N$ fixed.  The work of Mantilla-Soler combines the methods of the present paper with substantially more difficult group theory arising from the non-abelianness of the covers $X(Np^m)/X(N)$.  It would be very interesting to revisit his work from the Iwasawa-theoretic viewpoint described above.

\subsection{Acknowledgments}
We are very grateful to Nigel Boston, John Coates, Guillermo Mantilla-Soler, William McCallum, and Pavlos Tzermias for useful discussions about the material in this paper.  The author was partially supported by NSF-CAREER Grant DMS-0448750 and a Sloan Research Fellowship.

\section{Unipotent Galois actions on fundamental groups}

Let $X/F$ be defined as in the previous section, and write $\pi$ for the geometric etale fundamental group $\pi_1^{et}(X/\bar{F})$.  Let $T$ be a finite set of primes including $p,\infty$, and suppose 
\beq
\ldots \ra X_3/F \ra X_2/F \ra X_1/F \ra X_0/F = X
\eeq
is a pro-$p$ tower with good reduction away from $T$.

Write $\Theta^n$ for the Galois group of $(X_n)_{\bar{F}} / X_{\bar{F}}$, and $\Theta$ for the inverse limit of the $\Theta^n$.  Then $\Theta$ is a pro-$p$ group admitting a surjection
\beq
\phi: \pi \ra \Theta
\eeq
and an action of $G_F$, arising from the compatible actions of $G_F$ on the $\Theta^n$.

Furthermore, $\Theta$ admits a filtration by finite-index normal subgroups
\beq
\Theta = \Theta_0 \supset \Theta_1 \supset \Theta_2 \supset \ldots
\eeq
where $\Theta_n$ is the kernel of the projection from $\Theta$ to $\Theta^n$.  Write $H_n$ for $\phi^{-1}(\Theta_n)$.  Then $H_n = \pi_1^{et}(X_n / \bar{F})$.

Let $V_n$ be the group $H_n^{ab}/pH_n^{ab}$.  Then $V_n$ is a
finite-dimensional vector space over $\Fp$.  Note that the action of $G_F$ on $V_n$ factors through $G_T(F)$, by the good reduction hypothesis.

\begin{prop} Let $V$ be a subquotient of $V_n$.  Then 
\beq
\dim_{\Fp} H^1(G_T(F),V) \leq (\dim_{\Fp} V)(\dim_{\Fp}
H^1(G_T(F),\F_p)). 
\eeq
\label{pr:h1bound}
\end{prop}

\begin{proof}  We begin with a group-theoretic lemma.

\begin{lem} The action of $G_F$ on $V_n$ is unipotent.
\label{le:unipotent}
\end{lem}

\begin{proof}  The image of $G_F$ on $\pi$ lies in the kernel of
\beq
\Aut(\pi) \ra \Aut(\pi^{ab} / p \pi^{ab})
\eeq
which is a pro-$p$ group.  (See, e.g., \cite[Th 12.2.2]{hall:groups}.)   So the same is true for the image of $G_F$ in $GL(V_n)$.  But an action of a $p$-group on a finite-dimensional vector space over $\F_p$ is automatically unipotent.
\end{proof}

By Lemma~\ref{le:unipotent}, we have a filtration of $G_F$-modules (whence also $G_T(F)$-modules)
\beq
V_n = V_{n;0} \supset V_{n;1} \supset \ldots V_{n;N} = 0.
\eeq
whose successive terms yield exact sequences 
\beq
0 \ra V_{n;i+1} \ra V_{n;i} \ra (\Z/p\Z)^{r_i} \ra 0.
\eeq
We thus obtain a cohomology sequence
\beq
H^1(G_T(F),V_{n;i+1}) \ra H^1(G_T(F),V_{n;i})
\ra H^1(G_T(F),(\Z/p\Z)^{r_i}).
\eeq
It follows that
\beq
\dim_{\Fp} H^1(G_T(F),V_{n;i}) \leq 
\dim_{\Fp} H^1(G_T(F),V_{n;i+1}) + r_i \dim_{\Fp} H^1(G_T(F),\Z/p\Z)
\eeq
and, by induction,
\beq
\dim_{\Fp} H^1(G_T(F),V_{n;i}) \leq (\sum_{j \geq i} r_j)\dim_{\Fp}
H^1(G_T(F),\Z/p\Z).
\eeq
The proposition follows by setting $i=0$.
\end{proof}

\section{Towers of curves}

We now explain how Proposition~\ref{pr:h1bound} can be used to give
upper bounds for Mordell-Weil ranks (more precisely, Selmer ranks) in towers of curves.  Let
\beq
\ldots \ra X_3/F \ra X_2/F \ra X_1/F \ra X_0/F = X
\eeq
be a pro-$p$ tower with good reduction away from $T$, and define $\pi,\Theta_n,H_n,V_n$ as in the previous section.

\begin{thm}  Let $\tilde{X}_n/F$ be a smooth proper curve over $F$ containing $X_n$ as open subscheme, let $J_n$ be the Jacobian of $\tilde{X}_n$, and let $A$ be a subquotient of $J_n$.  Suppose that $G_F$ acts trivially on $H_1(X,\F_p)$.  Then
\beq
\frac{\rank A(F)}{\dim A} \leq 2\dim_{\Fp} H^1(G_T(F),\F_p).
\eeq
\label{th:main}
\end{thm}

\begin{proof}  The usual descent on abelian varieties shows that
\beq
\rank A(F) \leq \dim_{\Fp} A(F)/pA(F) \leq \dim_{\Fp}
H^1(G_T(F), A[p]).
\eeq

The $G_T(F)$-module $J_n[p]$ parametrizes \'{e}tale
abelian covers of $\tilde{X}_n$ with exponent $p$; these restrict
to \'{e}tale abelian covers of $X_n$ with exponent $p$.  It follows that $J_n[p]$ is a
quotient of $V_n$, whence $A[p]$ is a subquotient of $V_n$.  The result now follows immediately from Proposition~\ref{pr:h1bound}.

\end{proof}

We give some examples where Theorem~\ref{th:main} applies.

\begin{exmp}
Let $X/F$ be a proper curve of genus $g \geq 2$, admitting a smooth model over $\Spec \OO_F[1/T]$, and suppose $P \in X(F)$ is a rational point.  Define $X_n$ to be the Cartesian product
\beq
\begin{CD}
X_n @>>> J(X) \\
@VVV @VV[p^n]V \\
X @>>> J(X)
\end{CD}
\eeq
where the morphism $X \ra J(X)$ is the Abel-Jacobi map sending $Q$ to
$[Q-P]$, and $[p^n]$ is multiplication by $p^n$.  Then the $X_n$ form a pro-$p$ tower with good reduction away from $T$.  In this case, $\Theta = \pi^{ab} \cong \Zp^{2g}$.  The condition that $X$ is proper can be removed; see example~\ref{ex:fermat} below for the simplest case.
\label{ex:closedcurve}
\end{exmp}

\begin{exmp}
Let $X = Y(p)$ be the moduli space of elliptic curves with full level $p$ structure, and let $X_n = Y(p^{n+1})$.  Then the $X_n$ form a pro-$p$ tower with good reduction away from $T=p\infty$.  Here, $\Theta = \Gamma(p)/ \pm 1$, where $\Gamma(p)$ is the level $p$ congruence subgroup of $\SL_2(\Z_p)$.  This case is treated in detail by Mantilla-Soler in \cite{mant:thesis}. 
\end{exmp}

\begin{exmp}
\label{ex:fermat}
Let $X$ be $\P^1 - \set{0,1,\infty}$, which we write as the complement of the coordinate axes in the projective curve $x + y + z = 0$.  Let $X_n$ be the complement of the coordinate axes in the Fermat curve
\beq
x^{p^n} + y^{p^n} + z^{p^n} = 0.
\label{eq:fermat}
\eeq
Then the natural projection $X_n  \ra X$ is geometrically Galois with group $(\Z/p^n\Z)^2$, and the $X_n$ form a pro-$p$ tower with good reduction away from $T = p\infty$.  In this case, $\Theta = \Z_p^2$.  Theorem~\ref{th:main} thus tells us that
\beq
\rank J_n(\Q(\zeta_p)) \leq 2 \dim_{\F_p}H^1(G_T(\Q(\zeta_p)),\F_p) \dim J_n.
\eeq
The constant on the right-hand side depends critically on the arithmetic of $p$; in particular, it is large when the cyclotomic extension $\Q(\zeta_p)$ has a large $p$-class group.  In the following section we will explain how to improve the above bound to one with milder dependence on $p$.
\end{exmp}

\section{Towers of Fermat curves}

In this section we prove the following theorem.

\begin{thm} Let $p$ be an odd prime, and let $K$ be a number field such that the cyclotomic
extension $K(\zeta_{p^\infty})/K(\zeta_p)$ has $\mu$-invariant $0$. Let $J_n$ be the Jacobian of the Fermat curve $x^{p^n} + y^{p^n} + z^{p^n} = 0$.  Then there exists a constant
$C'$, depending on $K$ and $p$, such that
\beq
\rank_\Z  J_n(K) \leq [K:\Q]\dim J_n + C' p^n.
\eeq
\label{th:fermatrank}
\end{thm}

Let $K$ be a number field, and let $X/K, X_n/K$ be as in Example~\ref{ex:fermat}.  Let $F = K(\zeta_p)$, and let $T$ be the set of primes dividing
$p\infty$.   We are now in the situation of
Theorem~\ref{th:main}, which tells us that 
\beq
\rank J_n(F) \leq 2 \dim J_n \dim_{\Fp} H^1(G_T(F),\F_p).
\eeq

In other words, we have shown that the rank of $J_n$ over $F$ is bounded by a
constant multiple $c$ of its dimension, where $c$ depends on
$p$.  We now turn to the more delicate problem of showing that $c$ can
be chosen independently from $p$.

Let $F_n = F(\mu_{p^n})$.  Since Theorem~\ref{th:main} applies over
any field containing $\mu_p$, we also have
\beq
\rank J_n(F_n) \leq 2 \dim J_n \dim_{\Fp}
H^1(G_T(F_n)),\F_p).
\eeq

We have an inflation-restriction exact sequence
\beq
0 \ra H^1(\Gal(F_\infty/F_n),\F_p) \ra
H^1(G_T(F_n),\F_p) \ra H^1(G_T(F_\infty),\F_p)^{\Gal(F_\infty/F_n)} \ra
H^2(\Gal(F_\infty/F_n),\F_p)
\eeq
in which the last term is $0$, since $\Gal(F_\infty/F_n) \cong \Zp$ has
cohomological dimension $1$.  So
\begin{equation}
\dim_{\Fp} H^1(G_T(F_n)),\F_p) = 1 + \dim_{\Fp}
H^1(G_T(F_\infty),\F_p)^{G(F_\infty/F_n)}.
\label{eq:h1fnmodp}
\end{equation}
Furthermore, the long exact sequence
\beq
0 = H^0(G_T(F_\infty),\Qp/\Zp) / p   \ra H^1(G_T(F_\infty),\Z/p\Z) \ra H^1(G_T(F_\infty),\Qp/\Zp)[p] \ra 0
\eeq
implies that $H^1(G_T(F_\infty),\F_p) =
H^1(G_T(F_\infty),\Qp/\Zp)[p]$.

Now the group $H^1(G_T(F_\infty),\Qp/\Zp))$ is a well-understood
object of classical Iwasawa theory.  It is a cofinitely generated $\Zp[[G(F_\infty/F)]]$-module pseudo-isomorphic to 
\beq
M^{r_2(F)} \oplus T
\eeq
where $M$ is cofree and $T$ is cotorsion, and $r_2(F)$ denotes the number of complex places of $F$.  Suppose the cyclotomic extension $F_\infty/F$ has $\mu$-invariant $0$ (as is the case, for instance, whenever $F/\Q$ is abelian by the theorem of Ferrero and Washington.)  Then it follows from \cite[11.3.16,11.3.17]{neuk:cohomology} that $T[p]$ is a finite group.  We assume that the $\mu$-invariant of $F_\infty/F$ is $0$ from now on.

Since $F$ contains $\mu_p$, it is totally complex, and $r_2(F) =(1/2)[F:\Q]$.  It follows that $H^1(G_T(\F_\infty),\F_p)$ is pseudo-isomorphic to $M[p]^{(1/2)[F:\Q]}$.

Combining this with \eqref{eq:h1fnmodp} and the fact that
\beq
M[p]^{\Gal(F_\infty/F_n)} = \F_p[\Gal(F_n/F)]
\eeq
we have
\beq
\dim_{\Fp} H^1(G_T(F_n),\Z/p\Z)) \leq  (1/2)([F_n:\Q] + C)
\eeq
for some constant $C$ independent of $n$.

Applying Theorem~\ref{th:main}, we now have:

\begin{prop} Let $F \supset \Q(\zeta_p)$ be a number field whose
cyclotomic extension has $\mu$-invariant
$0$, and let $A/F$ be a subquotient of the Jacobian of the Fermat
curve $x^{p^n} + y^{p^n} + z^{p^n} = 0$, and let $F_n = F(\mu_{p^n})$.  Then
\beq
\rank A(F_n) \leq (\dim A)([F_n:\Q] + C)
\eeq
where $C$ is a constant independent of $n$.
\label{pr:fnrank}
\end{prop}

\medskip

Recall that we have denoted the Galois group of $X_n / \bar{K} \ra X / \bar{K}$ by $\Theta^n$, which is isomorphic to $(\Z/p^n\Z)^2$.  Write $H$ for $\Gal(F_n/K)$, and $G$ for the semidirect product $\Theta^n \rtimes H$ where $H$ is given its natural action on $\Theta^n \subset \Aut(X_n / \bar{K})$.  Then $J_n(F_n)$ carries an action of $G$, and the vector space $W = J_n(F_n) \tensor_\Z \C$ is a complex representation of $G$.  Moreover,
\beq
\rank J_n(F) = \dim W^H.
\eeq

Since $G$ is a semidirect product of an abelian group $\Theta^n$
by a subgroup of $\Aut(\Theta_n)$, its irreducible
representations are easy to describe (see \cite[8.2]{serr:lrfg}.) To be
precise:  let $\chi_i$ be a character of $\Theta^n$, and let $H_i
\subset H$ be the subgroup fixing $\chi_i$.  Then $\chi_i$
naturally extends to a character $\alpha_i$ of $G_i = \Theta^n H_i$.
Let $\psi$ be a character of $H_i$.  Then $\Ind_{G_i}^G (\alpha_i
\tensor \psi)$ is an irreducible representation of $G$, denoted
$\rho_{i,\psi}$.  The dimension of $\rho_{i,\psi}$ is $[H:H_i]$, and the $\rho_{i,\psi}$ comprise all irreducible representations of $G$.

Note that
\beq
\Res_H \rho_{i,\psi} = \Res_H \Ind^G_{G_i} (\alpha_i \tensor \psi)
= \Ind^H_{H_i} \psi
\eeq
by \cite[7.3]{serr:lrfg}.
So $\dim (\rho_{i,\psi})^H$ is $1$ if $\psi$ is trivial, and
$0$ otherwise.  In any case,
\begin{equation}
\dim (\rho_{i,\psi})^H \leq ([H:H_i]^{-1})\dim \rho_{i,\psi}.
\label{eq:hfixedbound}
\end{equation}

Let $B$ be the kernel of the map $\Theta^n \ra \Theta^{n-1}$.  Then we have an exact sequence of $\C[H]$-modules
\beq
0 \ra W^B \ra W \ra W' \ra 0.
\eeq
We note, first of all, that $W^B$ is precisely $J_{n-1}(F_n)
\tensor_\Z \C$.  The map $X_n \ra X_{n-1}$ induces a homomorphism
$J_{n-1} \ra J_n$, whose cokernel we call $J'_n$.  We also write $J'_0
= J$. Then $W'$ is isomorphic to $J'_n(F_n) \tensor_\Z \C$.

Now $W'$ is the direct sum of all the irreducible constituents of $W$ on which $B$ does not act trivially.  The $\rho_{i,\psi}$ on which $B$ does not act trivially are precisely those for which $\chi_i$ has exact order $p^n$.  Note that the stabilizer $H_i$ is trivial for any such $\chi_i$.  Applying \eqref{eq:hfixedbound} to each constituent of $W'$, we find
that
\beq
\dim (W')^H \leq |H|^{-1} \dim W' = [F_n:K]^{-1} \dim W'.
\eeq

Combining this with the upper bound on $\dim W'$ given by
Proposition~\ref{pr:fnrank}, we have
\beq
\dim (J'_n(K) \tensor_\Z \C) = \dim (W')^H 
\leq
([K:\Q] + C [F_n:K]^{-1})(\dim J'_n).
\eeq
Now $J_n$ is isogenous to $\oplus_{i=0}^n J'_i$.  It follows that
\beq
\rank J_n(K) \leq [K:\Q]\dim J_n + C \sum_{i=0}^n [F_i:K]^{-1} \dim
J'_i.
\eeq
The sum over $i$ on the right-hand side is bounded above by a constant multiple of $p^n$.   This completes the proof of Theorem~\ref{th:fermatrank}.

\medskip

\bibliographystyle{plain}
\bibliography{CurveRanks-ARXIV}

\end{document}